\documentclass[12pt]{article}
\usepackage{amsmath,amsthm,amssymb,euscript,verbatim}
\newtheorem{theorem}{Theorem}[section]

\newtheorem{corollary}{Corollary}[section]

\usepackage{makeidx}         
\usepackage{graphicx}        
\usepackage{multicol}        
\usepackage[bottom]{footmisc}

\makeindex             
\usepackage{latexsym}
\usepackage{epic,eepic,pstricks}

\begin{document}
\title
{\bf  Constructions of Compact Dupin Hypersurfaces with Non-constant Lie Curvatures}
\author
{Thomas E. Cecil}
\maketitle

\begin{abstract}
A hypersurface $M$ in the unit sphere $S^n \subset {\bf R}^{n+1}$ is Dupin if 
along each curvature surface of $M$, the corresponding principal curvature is constant.  If
the number $g$ of distinct principal curvatures is constant on $M$, then $M$ is called proper Dupin.
In this expository paper, we give a detailed description of two important types of constructions of compact proper Dupin hypersurfaces  in $S^n$. One construction was published in 1989 by Pinkall and Thorbergsson \cite{PT1}, and the second
was published in 1989 by Miyaoka and Ozawa \cite{MO}.  Both types of examples
have the property that they do not have constant Lie curvatures (Lie invariants discovered by Miyaoka \cite{Mi3}), which are the cross-ratios of the principal curvatures, taken four at a time.  Thus, these examples are not equivalent by a Lie sphere transformation to an isoparametric 
(constant principal curvatures) 
hypersurface in $S^n$.  So they are counterexamples to a conjecture of Cecil and Ryan \cite[p. 184]{CR7} in 1985
that every compact proper Dupin hypersurface 
in $S^n$ is equivalent to an isoparametric hypersurface by a Lie sphere transformation. 
\end{abstract}

\noindent
In this expository paper, we give a detailed description of two important types of constructions of compact proper Dupin hypersurfaces  in the unit sphere $S^n \subset {\bf R}^{n+1}$.  One construction was published in 1989 by Pinkall and Thorbergsson \cite{PT1}, and the second
was published in 1989 by Miyaoka and Ozawa \cite{MO}.   Both types of examples
have the property that they do not have constant Lie curvatures  (Lie invariants discovered by Miyaoka \cite{Mi3}), 
which are the cross-ratios of the principal curvatures, taken four at a time.  Thus, these examples are not equivalent by a Lie sphere transformation to an isoparametric 
hypersurface in $S^n$.  So they are counterexamples to a conjecture of Cecil and Ryan \cite[p. 184]{CR7} in 1985
that every compact proper Dupin hypersurface 
in $S^n$ is equivalent to an isoparametric hypersurface by a Lie sphere transformation.

In Section \ref{sec:1}, we give an introduction to the theory of Dupin hypersurfaces and the relevant results regarding compact proper Dupin hypersurfaces, including the relationship between Dupin and isoparametric hypersurfaces.  We also define the key notion of the Lie curvatures of a proper Dupin hypersurface.

In Section \ref{sec:2}, we give a brief description of the method for studying submanifolds of $S^n$ or Euclidean space
${\bf R}^n$ in the context of Lie sphere geometry.  This is necessary to define concepts such as ``Lie equivalent'' precisely, but the actual constructions of the examples in Section \ref{sec:3.7} do not require much use of Lie sphere geometry.

In Section \ref{sec:3.7}, we describe the constructions of Pinkall and Thorbergsson \cite{PT1}, and  Miyaoka and
Ozawa \cite{MO}, of compact proper Dupin hypersurfaces with non-constant Lie curvatures.  

The examples of 
Pinkall-Thorbergsson have $g=4$ principal curvatures, and they 
are certain types of deformations of OT-FKM-type isoparametric hypersurfaces (due to Ozeki-Takeuchi 
\cite{OT}--\cite{OT2}, and
Ferus-Karcher-M\"{u}nzner  \cite{FKM}), which have $g=4$ principal curvatures.  

The construction of Miyaoka-Ozawa yields examples with $g=4$, and also examples with $g=6$.
These examples are all hypersurfaces of the form $M^6 = h^{-1}(W^3) \subset S^7$,
where $h:S^7 \rightarrow S^4$ is the Hopf fibration, and $W^3 \subset S^4$
is a compact proper Dupin hypersurface that is not isoparametric.  If $W^3$ has 2 principal curvatures, then $M^6$ has 4 principal curvatures, and if $W^3$ has 3 principal curvatures, then $M^6$ has 6 principal curvatures.

In these notes, we will present both of these constructions in detail, following the treatment of these examples in the author's book \cite[pp. 112--123]{Cec1} very closely, and some passages are taken directly from that book.

\section{Dupin hypersurfaces}
\label{sec:1}
We first recall some basic definitions and results in the theory of Dupin hypersurfaces in $S^n$.
Let $f:M \rightarrow S^n$ be an oriented hypersurface with field of unit normals $\xi:M \rightarrow S^n$.  The 
{\em shape operator}
of $f$ at a point $x \in M$ is the symmetric linear transformation $A:T_xM \rightarrow T_xM$
defined by the equation
\begin{equation}
\label{eq:3.4.1}
df(AX) = - d\xi(X), \quad X \in T_xM.
\end{equation} 
The eigenvalues of $A$ are called the 
{\em principal curvatures}, and the corresponding eigenvectors are
called the {\em principal vectors}.  

We next recall the notion of a focal point  of an immersion.  For each
real number $t$, define a map
\begin{displaymath}
f_t:M \rightarrow S^n,
\end{displaymath}
by
\begin{equation}
\label{eq:3.4.2}
f_t = \cos t \ f + \sin t \ \xi.
\end{equation} 
For each $x \in M$, the point $f_t(x)$ lies an oriented distance $t$ along the normal geodesic
to $f(M)$ at $f(x)$.  A point $p = f_t(x)$ is called a 
{\em focal point of multiplicity} $m>0$ 
{\em of} $f$ {\em at} $x$ if the nullity of $df_t$ is equal to $m$ at $x$.  Geometrically, one thinks of focal points
as points where nearby normal geodesics intersect.  

It is well known that the location of focal points is related to the
principal curvatures.  Specifically, if $X \in T_xM$, then by equation \eqref{eq:3.4.1} we have
\begin{equation}
\label{eq:3.4.3}
df_t(X) = \cos t \ df(X) + \sin t \ d\xi(X) = df(\cos t \ X - \sin t \ AX).
\end{equation} 
Thus, $df_t(X)$ equals zero for $X\neq0$ if and only if $\cot t$ is a principal curvature of $f$ at $x$,
and $X$ is a corresponding principal vector.  Hence, $p = f_t(x)$ is a focal point of $f$ at $x$ of
multiplicity $m$ if and only if $\cot t$ is a principal curvature of multiplicity $m$ at $x$.  Note
that each principal curvature 
\begin{displaymath}
\kappa = \cot t, \quad 0<t<\pi,
\end{displaymath}
produces two distinct antipodal focal points
on the normal geodesic with parameter values $t$ and $t+\pi$.  

The hypersphere in $S^n$ centered at a focal point
$p$ and tangent to $f(M)$ at $f(x)$ is called a 
{\em curvature sphere} of $f$ at $x$.  The
two antipodal focal points determined by $\kappa$ are the two centers of the corresponding curvature sphere in $S^n$.
Thus, the correspondence between principal curvatures and curvature spheres is bijective.  The multiplicity
of the curvature sphere is by definition equal to the multiplicity of the corresponding principal curvature.

A {\em curvature surface} of $M$ is a smooth submanifold $S\subset M$
such that for each point $x \in S$, the tangent space
$T_xS$ is equal to a principal space (i.e., an eigenspace) of the shape operator
$A$ of $M$ at $x$. This generalizes the classical notion of a line of curvature 
for a principal curvature of multiplicity one.

A hypersurface $M$ is said to be {\em Dupin} if:

\begin{enumerate}
\item[(a)] along each curvature surface, the corresponding principal curvature is constant.
\end{enumerate}
Furthermore, a Dupin hypersurface $M$ is called {\em proper Dupin} if, in addition to Condition (a),
the following condition is satisfied:
\begin{enumerate}
\item[(b)] the number $g$ of distinct principal curvatures is constant on
$M$.
\end{enumerate}

These conditions are preserved by stereographic projection (see, for example, \cite[pp. 20--24]{CR8}), 
so this theory is essentially the same for hypersurfaces in 
$S^n$ or ${\bf R}^n$, and we will use whichever of the two ambient spaces is more convenient
in describing the examples.

It follows from the Codazzi equation that if a continuous principal curvature function $\mu$ has constant
multiplicity $m$ on a connected open subset $U \subset M$, then $\mu$ is a smooth function, and
the distribution $T_{\mu}$ of principal
spaces corresponding to $\mu$ is a smooth foliation whose leaves are the curvature surfaces corresponding to 
$\mu$ on $U$. This principal curvature function
$\mu$ is constant along each of its curvature surfaces in $U$
if and only if these curvature surfaces are open subsets
of $m$-dimensional great or small spheres in $S^n$ (see, for example, \cite[pp. 18--32]{CR8}).

There exist many examples of Dupin hypersurfaces that are not proper Dupin, because the number of 
distinct principal curvatures is not constant on the hypersurface.
This also results in curvature surfaces that are not leaves of a 
principal foliation.  For example, Pinkall \cite{P4} showed that
a tube $M^3$ in ${\bf R}^4$ of constant radius over a torus of revolution 
$T^2 \subset {\bf R}^3 \subset {\bf R}^4$ is Dupin, but not proper Dupin.
It has three distinct principal curvatures at all points, except for points on two 2-dimensional tori, at which 
there are only two distinct principal curvatures.
(See also \cite[p. 69]{Cec1}.)

An important concept that is closely related to the Dupin condition is the notion of a taut embedding
(see, for example, \cite{CR7}).
An embedding $\phi:V \rightarrow S^n$ 
of a compact, connected manifold $V$ into the unit sphere $S^n \subset {\bf R}^{n+1}$ 
is said to be {\em taut}, if every nondegenerate spherical distance function $d_p$, $p \in S^n$, is a perfect
Morse function on $V$, i.e., it 
has the minimum
number of critical points  on $V$ required by the Morse inequalities. 

In particular, Thorbergsson \cite{Th1} proved that a compact, proper Dupin hypersurface $M \subset S^n$ is taut.
(See \cite[pp. 65--74]{CR8} for more discussion of the relationship
between the concepts of Dupin and taut.) Several important facts about this relationship 
are used in Section \ref{sec:3.7} in describing the construction of 
compact proper Dupin hypersurfaces due to Miyaoka and Ozawa.

Clearly isoparametric  hypersurfaces (which have constant principal curvatures) in $S^n$ are proper Dupin, and so are
those hypersurfaces in ${\bf R}^n$ obtained from isoparametric hypersurfaces in $S^n$ via
stereographic projection (see, for example, \cite[pp. 28--30]{CR8}).  
In particular, the well-known ring cyclides of Dupin in ${\bf R}^3$ are obtained
in this way from a standard product torus $S^1(r) \times S^1(s)$ in $S^3$, where $r^2+s^2=1$.

The first examples of compact proper Dupin hypersurfaces are those that are equivalent by a Lie sphere transformation
to an isoparametric hypersurface in the sphere $S^n$.  (See Section \ref{sec:2} for more detail on the notion of Lie equivalence of hypersurfaces.)

M\"{u}nzner
\cite{Mu}--\cite{Mu2} showed that the number $g$ of distinct principal curvatures of an 
isoparametric hypersurface
in $S^n$ must be $1,2,3,4$ or 6.  Thorbergsson \cite{Th1} then
showed that the same restriction holds for a compact proper Dupin hypersurface $M^{n-1}$ embedded in $S^n$.  
He first proved that $M^{n-1}$ must be taut
in $S^n$. Using tautness, he then showed that $M^{n-1}$ divides $S^n$
into two ball bundles over the first focal submanifolds on either side of $M^{n-1}$ in $S^n$.  This topological data
is all that is required for M\"{u}nzner's restriction on $g$.

Compact proper Dupin hypersurfaces in $S^n$ have been classified in the cases $g = 1,2$ and 3.  In each case, 
the hypersurface $M^{n-1}$ must be Lie equivalent 
to an isoparametric hypersurface.  The case $g = 1$ is simply the case of umbilic
hypersurfaces, 
i.e., hyperspheres in $S^n$.  In the case $g=2$, Cecil and Ryan \cite{CR2} showed that $M^{n-1}$
must be a cyclide of Dupin, and thus it is M\"{o}bius equivalent to a standard product of spheres
\begin{displaymath}
S^k (r) \times S^{n-1-k} (s) \subset S^n (1) \subset {\bf R}^{n+1}, \quad r^2 + s^2 = 1.
\end{displaymath}

In the case $g=3$, Miyaoka
\cite{Mi1} proved that $M^{n-1}$ must be Lie equivalent (although not necessarily M\"{o}bius equivalent)
to an isoparametric hypersurface
(see also \cite{CCJ2}).  Cartan \cite{Car3} had shown previously that an isoparametric hypersurface with $g=3$
principal curvatures is a tube over a standard 
embedding of a projective plane ${\bf FP}^2$,
for ${\bf F} = {\bf R}, {\bf C}, {\bf H}$ (quaternions) or ${\bf O}$ (Cayley numbers), 
in $S^4, S^7, S^{13}$
and $S^{25}$, respectively. 

These results led to the conjecture of Cecil and Ryan \cite[p. 184]{CR7} in 1985
that every compact connected proper Dupin hypersurface embedded in $S^n$ is Lie equivalent to an isoparametric
hypersurface. 

For most of the 1980's, all attempts to verify this conjecture in the cases of $g=4$ and 6 principal curvatures were unsuccessful.  
Finally, in papers published in 1989,
Pinkall and Thorbergsson 
\cite{PT1}, and Miyaoka and Ozawa \cite{MO}, gave two different constructions for producing
counterexamples to the conjecture with $g =4$ principal curvatures.  The method of Miyaoka and Ozawa also yields
counterexamples to the conjecture in the case $g=6$.  (See the author's survey article \cite{Cec10} for a description of progress on the conjecture of Cecil and Ryan.)

For both of these constructions, it is important to consider the Lie curvatures of the hypersurface.  These
are the cross-ratios of the principal curvatures taken four at a time. Specifically,
for four distinct numbers $a,b,c,d$ in ${\bf R} \cup \{\infty\}$, we adopt the notation
\begin{equation}
\label{eq:3.5.4}
[a,b;c,d] = \frac{(a-b)(d-c)}{(a-c)(d-b)}
\end{equation} 
for the cross-ratio of $a,b,c,d$.  We use the usual conventions involving operations with $\infty$.
For example, if $d= \infty$, then the expression $(d-c)/(d-b)$ evaluates to one,
and the cross-ratio $[a,b;c,d]$ equals $(a-b)/(a-c)$.

Let $M^{n-1} \subset S^n$ be a proper Dupin hypersurface, and suppose
that $\kappa_1,\ldots,\kappa_g, \ g \geq 4,$ are the distinct principal curvatures of $M^{n-1}$ at a point
$x \in M^{n-1}$. Then for any choice of four numbers $h,i,j,k$ from the set $\{1,\ldots,g\}$, the cross-ratio
\begin{equation}
\label{eq:3.5.5}
[\kappa_h,\kappa_i;\kappa_j, \kappa_k]
\end{equation} 
is called a {\em Lie curvature} of $M^{n-1}$ at $x$.
The Lie curvatures were discovered by Miyaoka \cite{Mi3}, who proved Miyaoka's Theorem that the Lie curvatures
are naturally invariant under the action of a Lie sphere transformation on a hypersurface, as we will discuss in Section \ref{sec:2}.

The constructions of Pinkall-Thorbergsson and Miyaoka-Ozawa produce examples where 
the Lie curvatures are not constant, and so these examples cannot be Lie  equivalent to an isoparametric hypersurface, which must have constant Lie curvatures,
since the principal curvatures themselves are constant.

\section{Submanifolds in Lie sphere geometry}
\label{sec:2}
We now give a brief description of the method for studying submanifolds
in $S^n$ or ${\bf R}^n$ within the context of Lie sphere geometry
(introduced by Lie \cite{Lie}).  The reader is referred to the papers of
Pinkall \cite{P4}, Chern \cite{Chern}, Cecil and Chern \cite{CC1}--\cite{CC2}, Cecil and Jensen \cite{CJ2},
or the books of Cecil \cite{Cec1}, or
Jensen, Musso and Nicolodi \cite{Jensen-Musso-Nicolodi}, for more detail.

Lie sphere geometry is situated in real projective space ${\bf P}^n$,
so we now briefly review some concepts and notation from projective geometry.
We define an equivalence relation on ${\bf R}^{n+1} - \{0\}$ by setting
$x \simeq y$ if $x = ty$ for some nonzero real number $t$.  We denote
the equivalence class determined by a vector $x$ by $[x]$.  Projective
space ${\bf P}^n$ is the set of such equivalence classes, and it can
naturally be identified with the space of all lines through the origin
in ${\bf R}^{n+1}$.  The rectangular coordinates $(x_1, \ldots, x_{n+1})$
are called {\em homogeneous coordinates}
of the point $[x]$ in ${\bf P}^n$, and they
are only determined up to a nonzero scalar multiple.  

A {\em Lie sphere} in $S^n$ is an oriented hypersphere or a point sphere in $S^n$.
The set of all Lie spheres
is in bijective correspondence with the set of all points $[x] = [(x_1,\ldots,x_{n+3})]$
in projective space ${\bf P}^{n+2}$ that lie
on the quadric hypersurface $Q^{n+1}$ 
determined by the equation 
$\langle x, x \rangle = 0$, where
\begin{equation}
\label{Lie-metric}
\langle x, y \rangle = -x_1 y_1 + x_2 y_2 + \cdots +x_{n+2} y_{n+2} - x_{n+3} y_{n+3}
\end{equation}
is a bilinear form of signature $(n+1,2)$ on ${\bf R}^{n+3}$.  
We let  ${\bf R}^{n+3}_2$ denote ${\bf R}^{n+3}$ endowed with
the indefinite inner product \eqref{Lie-metric}.  The quadric 
$Q^{n+1}$ in  ${\bf P}^{n+2}$  is called the {\em Lie quadric}.
We denote the standard basis in ${\bf R}^{n+3}_2$  by $\{e_1,\ldots,e_{n+3}\}$, where $e_1$ and $e_{n+3}$
are timelike unit vectors, and $\{e_2,\ldots,e_{n+2}\}$ are spacelike unit vectors.

The specific correspondence is as follows.  We identify $S^n$ with
with the unit sphere in ${\bf R}^{n+1} \subset {\bf R}^{n+3}_2$, where ${\bf R}^{n+1}$ is spanned by the 
standard basis vectors $\{e_2,\ldots,e_{n+2}\}$ in ${\bf R}^{n+3}_2$.
Then the oriented hypersphere with center $p \in S^n$ and
signed radius $\rho$ corresponds to the point in $Q^{n+1}$ with homogeneous coordinates,
\begin{equation}
\label{eq:1.4.4}
(\cos \rho , p, \sin \rho ),
\end{equation}
where $ - \pi < \rho < \pi$.

We can designate the orientation of the sphere by the sign of
$\rho$ as follows.  A positive radius $\rho$ in 
\eqref{eq:1.4.4} corresponds to the orientation of the sphere given by 
the field of unit normals which are tangent vectors to geodesics in $S^n$ 
going from $-p$ to $p$, and a negative radius corresponds
to the opposite orientation.
Each oriented sphere can be considered in two ways, with center $p$ and signed radius $\rho, - \pi < \rho < \pi$,
or with center $-p$ and the appropriate signed radius $\rho \pm \pi$.
Point spheres $p$ in $S^n$ correspond to those points $[(1, p, 0)]$ in $Q^{n+1}$ with radius $\rho = 0$.

Due to the signature of the metric $\langle \ ,\ \rangle$, the
Lie quadric $Q^{n+1}$ contains projective lines but no linear
subspaces of ${\bf P}^{n+2}$ of higher dimension (see, for example, \cite[p. 21]{Cec1}).  
A straightforward calculation
shows that if $[x]$ and $[y]$ are two points on the quadric, 
then the line $[x,y]$ 
lies on $Q^{n+1}$ if and only if $\langle x,y \rangle=0$.  Geometrically,
this condition means that the hyperspheres in $S^n$ corresponding
to $[x]$ and $[y]$ are in oriented contact, i.e., they are tangent to each other
and have the same orientation at the point of contact.
For a point sphere and an oriented sphere, oriented contact means that the point lies on the sphere.
The 1-parameter family
of Lie spheres  in $S^n$ corresponding to the points on a line on the Lie quadric is called a 
{\em parabolic pencil of spheres}.

If we wish to work in ${\bf R}^n$, the set of {\em Lie spheres} consists of all oriented hyperspheres, oriented hyperplanes, and point spheres in ${\bf R}^n \cup \{ \infty\}$.  As in the spherical case,
we can find a bijective correspondence between the set of all Lie spheres and the set of
points on $Q^{n+1}$, and the notions of oriented contact and parabolic pencils of
Lie spheres are defined in a natural way (see, 
for example, \cite[pp. 14--23]{Cec1}).

A {\em Lie sphere transformation} 
is a projective transformation of ${\bf P}^{n+2}$ which maps the Lie quadric $Q^{n+1}$
to itself.  In terms of the geometry of  $S^n$ or ${\bf R}^n$, 
a Lie sphere transformation maps Lie spheres to Lie spheres.
Furthermore, since a Lie sphere transformation maps lines on $Q^{n+1}$ to lines on $Q^{n+1}$, 
a Lie sphere transformation preserves oriented contact of Lie spheres  (see Pinkall \cite[p. 431]{P4}
or \cite[pp. 25--30]{Cec1}).

The group of Lie sphere transformations is
isomorphic to $O(n+1,2)/ \{ \pm I \}$, where $O(n+1,2)$
is the orthogonal group for the metric in equation \eqref{Lie-metric}.
A Lie sphere transformation that takes point spheres to point spheres is a {\em M\"{o}bius transformation}, i.e.,
it is induced by a conformal diffeomorphism of $S^n$, and the set of all M\"{o}bius transformations is a subgroup of the Lie sphere group.  

An example of a Lie sphere transformation that is not
a M\"{o}bius transformation is a {\em parallel transformation} $P_t$ given by the formula,
\begin{eqnarray}
\label{eq:3.4.9}
P_t e_1 & = & \cos t \ e_1 + \sin t \ e_{n+3}, \nonumber \\
P_t e_{n+3} & = & - \sin t \ e_1 + \cos t \ e_{n+3},\\
P_t e_i & = & e_i, \quad 2 \leq i \leq n+2. \nonumber
\end{eqnarray}
The transformation $P_t$ has the effect of adding $t$ to the signed radius
of each oriented sphere in $S^n$ while keeping the center
fixed (see \cite[pp. 25--49]{Cec1}).

The $(2n-1)$-dimensional manifold $\Lambda^{2n-1}$ of projective lines on the quadric $Q^{n+1}$ has a 
contact structure, i.e., a
$1$-form $\omega$ such that $\omega \wedge (d\omega)^{n-1}$ does not vanish on $\Lambda^{2n-1}$.  The condition $\omega = 0$ defines a codimension one distribution $D$ on $\Lambda^{2n-1}$ which has 
integral submanifolds
of dimension $n-1$, but none of higher dimension.  Such an integral  submanifold 
$\lambda: M^{n-1} \rightarrow \Lambda^{2n-1}$ of $D$ of dimension $n-1$ is called a 
{\em Legendre submanifold} (see \cite[pp. 51--64]{Cec1}).

An oriented hypersurface $f:M^{n-1} \rightarrow S^n$ with field of unit
normals $\xi :M^{n-1} \rightarrow S^n$ naturally induces
a Legendre submanifold $\lambda = [k_1, k_2]$, where 
\begin{equation}
k_1 = (1,f,0), \quad k_2 = (0,\xi ,1),
\end{equation}
in homogeneous coordinates.
For each $x \in M^{n-1}, [k_1(x)]$ is the unique point sphere and
$[k_2 (x)]$ is the unique great sphere in the parabolic pencil of spheres in $S^n$ corresponding to 
the points on the line $\lambda (x)$.  The Legendre submanifold $\lambda$ is called the {\em Legendre lift}
of the oriented hypersurface $f$ with field of unit normals $\xi$.

The case of an immersed submanifold $\phi :V \rightarrow S^n$ of codimension 
greater than one is handled in a similar way.
In that case, let
$B^{n-1}$ be the unit normal bundle\index{unit normal bundle}\index{normal!bundle}
\index{normal!unit|see{unit normal bundle}}
 of the submanifold $\phi$.  Then $B^{n-1}$ can be considered to be the
submanifold of $V \times S^n$ given by
\begin{displaymath}
B^{n-1} = \{ (x,\xi)| \phi(x) \cdot \xi = 0, \ d\phi(X) \cdot \xi = 0,\  \mbox{\rm for all }X \in T_xV\},
\end{displaymath} 
where the dot indicates the Euclidean inner product on ${\bf R}^{n+1}$.

The {\em Legendre lift} of the submanifold $\phi$
is the map $\lambda: B^{n-1} \rightarrow \Lambda^{2n-1}$ defined by
\begin{equation}
\label{eq:3.3.2}
\lambda(x,\xi) = [k_1(x,\xi), k_2 (x,\xi)],
\end{equation} 
where
\begin{equation}
\label{eq:3.3.3}
k_1(x,\xi) = (1, \phi(x), 0), \quad k_2(x,\xi) = (0, \xi, 1).
\end{equation} 
Geometrically, $\lambda(x,\xi)$ is the line on the quadric $Q^{n+1}$ corresponding to the 
parabolic pencil
of spheres in $S^n$ in oriented contact at the point $(\phi(x),\xi)$ in the unit tangent bundle $T_1S^n$ of $S^n$.

If $\beta$  is a Lie sphere transformation, then $\beta$ maps lines on
$Q^{n+1}$ to lines on $Q^{n+1}$, so it naturally
induces a map $\widetilde{\beta }$ from 
$\Lambda ^{2n-1}$ to itself.  If $\lambda$ is a
Legendre submanifold, then one can show that $\widetilde{\beta}\lambda$
is also a Legendre submanifold, which is denoted as
$\beta \lambda$ for short.  These two Legendre
submanifolds are said to be {\em Lie equivalent}.
We will also say that two submanifolds of $S^n$ or ${\bf R}^n$
are Lie equivalent, if their corresponding Legendre lifts
are Lie equivalent. 
 
If $\beta$ is a M\"{o}bius transformation, then the two Legendre
submanifolds  $\lambda$ and $\beta \lambda$ are said to be {\em M\"{o}bius equivalent}.
Finally, if $\beta$ is the parallel transformation $P_t$ in \eqref{eq:3.4.9} and
$\lambda$ is the Legendre lift of an
oriented hypersurface $f:M \rightarrow S^n$, then
$P_t\lambda$ is the Legendre lift of the
parallel hypersurface $f_{-t}$ at oriented distance $-t$ from $f$ (see, for example, \cite[p.67]{Cec1}).

It is easy to generalize the definitions of Dupin and proper Dupin hypersurfaces in $S^n$ to the class of Legendre submanifolds in Lie sphere geometry.  We simply replace the notion of a principal curvature (which is not Lie invariant)
with the notion of an oriented curvature sphere, which is Lie invariant when it is appropriately defined in Lie sphere geometry.

We then say that a Legendre submanifold $\lambda: M^{n-1} \rightarrow \Lambda^{2n-1}$
is a {\em Dupin submanifold} if:

\begin{enumerate}
\item[(a)] along each curvature surface, the corresponding
curvature sphere map is constant.
\end{enumerate}
Furthermore, a Dupin submanifold $\lambda$ is called {\em proper Dupin} if, in addition 
to Condition (a), the following condition is satisfied:

\begin{enumerate}
\item[(b)] the number $g$ of distinct curvature spheres is constant on $M$.
\end{enumerate}

One can easily show that a Lie sphere transformation $\beta$ maps curvature spheres of $\lambda$ to curvature
spheres of $\beta \lambda$, and that Conditions (a) and (b) are preserved by $\beta$ (see \cite[pp.67--70]{Cec1}).  
Thus, both the Dupin and proper Dupin properties are invariant under Lie sphere transformations.

\section{Construction of the Examples}
\label{sec:3.7} 
In this section, we describe the examples of Pinkall and Thorbergsson \cite{PT1}, and Miyaoka and Ozawa \cite{MO},
in detail, and we show that they do not have constant Lie curvatures.

The construction of Pinkall and Thorbergsson begins with an isoparametric hypersurface $M$
with four principal curvatures of 
OT-FKM-type constructed by Ozeki-Takeuchi \cite{OT}--\cite{OT2}, and Ferus, Karcher, and M\"{u}nzner \cite{FKM}.
The key is to consider one of the focal submanifolds $M_{+}$ of the family of isoparametric hypersurfaces.
We first recall some of the details of the OT-FKM construction (see also \cite[pp. 95--112]{Cec1}).\\

\noindent
{\bf Construction of Isoparametric Hypersurfaces of OT-FKM-type}\\

In a paper published in 1981, Ferus, Karcher and M\"{u}nzner
\cite{FKM} constructed an infinite class of 
isoparametric hypersurfaces with $g=4$ principal curvatures that includes all examples with $g=4$, except
for two homogeneous examples.
This construction is based on representations
of Clifford algebras, and the classification of such representations is an important element in the construction.
The FKM construction is a generalization of an earlier construction of Ozeki and Takeuchi\index{Ozeki-Takeuchi}
\cite{OT}--\cite{OT2}, who also used representations of certain Clifford algebras.  Thus, we refer to these examples as isoparametric hypersurfaces of OT-FKM-type.

The construction in the paper of 
Ferus, Karcher, and M\"{u}nzner \cite{FKM} focuses on the Cartan-M\"{u}nzner  polynomials that define the hypersurfaces,
although many other properties of the hypersurfaces are thoroughly developed.  

The following approach to the OT-FKM construction was given in the paper of Pinkall and Thorbergsson \cite{PT1}.  In this approach, one begins with one of the focal submanifolds of the family of isoparametric hypersurfaces and then defines the hypersurfaces as tubes over the focal submanifold.

Specifically, one
starts with a representation of the Clifford algebra 
$C_{m-1}$ on ${\bf R}^l$ determined by $l \times l$ skew-symmetric matrices
\begin{displaymath} 
E_1,\ldots,E_{m-1}
\end{displaymath}
satisfying the equations,
\begin{equation} 
\label{eq:3.7.1}
E_i^2 = -I, \quad E_i E_j = - E_j E_i, \quad i \neq j, \quad 1 \leq i,j \leq m-1.
\end{equation}
Note that the $E_i$ must also be orthogonal (see  \cite[p. 98]{Cec1}).

Two vectors $u$ and $v$ in ${\bf R}^l$ are said to be 
{\em Clifford orthogonal} if
\begin{equation} 
\label{eq:3.7.2}
u \cdot v = E_1 u \cdot v = \cdots = E_{m-1} u \cdot v = 0.
\end{equation}
One of the focal submanifolds of the isoparametric family that we are constructing  is given by
\begin{equation} 
\label{eq:3.7.3}
M_{+} = \{(u,v) \in S^{2l-1} \mid |u|= |v| = \frac{1}{\sqrt{2}}, u \cdot v = 0, E_i u \cdot v = 0, 1 \leq i \leq m-1 \},
\end{equation}
where $S^{2l-1}$ is the unit sphere in ${\bf R}^{2l} = {\bf R}^l \times {\bf R}^l$.
This manifold $M_{+}$ is the Clifford--Stiefel
manifold $V_2 (C_{m-1})$ of Clifford orthogonal 2-frames of length
$1/\sqrt{2}$ in ${\bf R}^l$.  Note that $M_{+} = V_2 (C_{m-1})$ is a submanifold of codimension $m+1$ in $S^{2l-1}$.

Ferus, Karcher, and M\"{u}nzner \cite{FKM} (see also \cite[pp.109--110]{Cec1}) 
showed that for any unit normal $\xi$ at any point $(u,v) \in V_2 (C_{m-1})$,
the shape operator $A_\xi$ has three distinct principal curvatures 
\begin{equation} 
\label{eq:3.7.4}
\kappa_1 = -1, \quad \kappa_2 = 0, \quad \kappa_3 = 1,
\end{equation}
with respective multiplicities $l-m-1$, $m$, $l-m-1$.  

The submanifold $V_2 (C_{m-1})$ of codimension $m+1$ in $S^{2l-1}$ has a Legendre lift
defined on the unit normal bundle 
$B(V_2 (C_{m-1}))$ of $V_2 (C_{m-1})$ in $S^{2l-1}$, as defined in Section \ref{sec:2} (see \cite[pp. 60--61]{Cec1}). 
Since $V_2 (C_{m-1})$ has codimension $m+1$ in $S^{2l-1}$,
this Legendre lift has a fourth principal curvature $\kappa_4 = \infty$ of multiplicity $m$ at each point of  
$B(V_2 (C_{m-1}))$ (see Theorem 4.15 of \cite[p. 74]{Cec1}).

For a  proper Dupin submanifold with $g$ distinct principal curvatures, we can order the principal curvatures so that 
\begin{equation}
\label{eq:3.5.9a}
\kappa_1 < \cdots < \kappa_g.
\end{equation}
In the case $g=4$, this leads to a unique Lie curvature $\Psi$ as defined in equation \eqref{eq:3.5.4} by
\begin{equation}
\label{eq:3.5.10}
\Psi = [\kappa_1, \kappa_2; \kappa_3, \kappa_4] = 
(\kappa_1 - \kappa_2)(\kappa_4 - \kappa_3)/(\kappa_1 - \kappa_3)(\kappa_4 - \kappa_2).
\end{equation}
The ordering of the principal curvatures implies that $\Psi$ satisfies $0 < \Psi < 1$.

Since $\kappa_4 = \infty$, the Lie curvature
$\Psi$ at any point of $B(V_2 (C_{m-1}))$
equals 
\begin{equation}
\label{eq:3.7.5}
\Psi =  (-1 - 0)(\infty - 1) / (-1 - 1)(\infty - 0) = 1/2,
\end{equation}
by the usual conventions involving $\infty$ mentioned after \eqref{eq:3.5.4} (see \cite[pp.72--82]{Cec1}).

Since all four principal curvatures are constant on $B(V_2 (C_{m-1}))$, 
a tube $M_t$ of spherical radius $t$,
where $0 < t < \pi$ and $t\notin \{\frac{\pi}{4}, \frac{\pi}{2}, 
\frac{3 \pi}{4} \}$, over $V_2 (C_{m-1})$ is an isoparametric hypersurface with four distinct principal curvatures
(see Corollary 4.41 of \cite[p. 111]{Cec1}).  These hypersurfaces $M_t$ form a family
of isoparametric hypersurfaces of OT-FKM-type.

Note that M\"{u}nzner
\cite{Mu} determined the values of the principal curvatures of any isoparametric hypersurface $M$
in $S^n$ with four principal curvatures.  From M\"{u}nzner's results, one can show that the Lie curvature $\Psi = 1/2$ on all of $M$, for any isoparametric hypersurface $M$
in $S^n$ with four principal curvatures. \\

\noindent
{\bf Construction of the Pinkall-Thorbergsson Examples}\\

We now begin the construction of compact proper Dupin hypersurfaces with non-constant Lie curvatures
due to Pinkall and Thorbergsson \cite{PT1}.  
These hypersurfaces are deformations of the isoparametric hypersurfaces of OT-FKM-type described above, and they have $g=4$ principal curvatures.

Given positive real numbers $\alpha$ and $\beta$ with
\begin{displaymath}
\alpha^2 + \beta^2 = 1, \quad \alpha \neq \frac{1}{\sqrt{2}}, \quad \beta \neq \frac{1}{\sqrt{2}},
\end{displaymath}
let 
\begin{displaymath}
T_{\alpha,\beta} : {\bf R}^{2l} \rightarrow {\bf R}^{2l},
\end{displaymath}
be the linear map defined by
\begin{displaymath}
T_{\alpha,\beta} (u,v) = \sqrt{2}\  (\alpha u, \beta v).
\end{displaymath}
Then for $(u,v) \in V_2 (C_{m-1})$, we have
\begin{displaymath}
|T_{\alpha,\beta} (u,v)|^2 = 2 (\alpha^2 (u \cdot u) + \beta^2 (v \cdot v)) = 2 (\frac{\alpha^2}{2} + \frac{\beta^2}{2})=1,
\end{displaymath}
and thus the image $V^{\alpha,\beta}_2 = T_{\alpha,\beta}V_2 (C_{m-1})$ is a submanifold of $S^{2l-1}$ of
codimension $m+1$ also.

Furthermore, we can see that the Legendre lift of $V^{\alpha,\beta}_2$ is proper Dupin 
as follows.  (For the sake of brevity, we
will say that a submanifold $V \subset S^n$ of codimension greater than one is proper 
Dupin if its Legendre lift is proper Dupin.)

We use the notion of curvature surfaces of a submanifold of codimension greater than one 
defined by Reckziegel \cite{Reck2}.
Specifically, suppose that $V \subset S^n$ is a submanifold of codimension greater than one, and let $B(V)$
denote its unit normal bundle. 
A connected submanifold $S \subset V$ is called a curvature surface if there exits a 
parallel section $\eta:S \rightarrow B(V)$ such that for each $x \in S$, 
the tangent space $T_xS$ is equal to some eigenspace of $A_{\eta (x)}$. The corresponding principal curvature
$\kappa :S \rightarrow {\bf R}$ is then a smooth function on $S$.  Reckziegel showed that if a principal curvature
$\kappa$ has constant multiplicity $\mu$ on $B(V)$ and is constant along each of its curvature surfaces,
then each of its curvature surfaces is an open subset of a $\mu$-dimensional metric sphere in $S^n$.
Since our particular submanifold $V_2 (C_{m-1})$ is compact, all of the curvature surfaces of 
the principal curvatures $\kappa_1$,
$\kappa_2$ and $\kappa_3$ given in equation \eqref{eq:3.7.4} are spheres of the appropriate dimensions in $S^{2l-1}$.

We now show that the Legendre lift of $V^{\alpha,\beta}_2$ is proper Dupin with four principal curvatures
\begin{equation}
\label{eq:Lie-curv-lambda}
\lambda_1 < \lambda_2 < \lambda_3 < \lambda_4.
\end{equation}
Since $V^{\alpha,\beta}_2$ has codimension $m+1$, the principal curvature $\lambda_4 = \infty$ has multiplicity $m$
and is constant along its curvature surfaces. To complete the proof that $V^{\alpha,\beta}_2$ is proper Dupin,
we establish a bijective correspondence between the other curvature surfaces of $V_2 (C_{m-1})$ and those of
$V^{\alpha,\beta}_2$.  Let $S$ be any curvature surface of $V_2 (C_{m-1})$.  Since $V_2 (C_{m-1})$ is compact
and proper Dupin, $S$ is a $\mu$-dimensional sphere, where $\mu$ is the multiplicity of the corresponding
principal curvature of $V_2 (C_{m-1})$.  Along the curvature surface $S$, the corresponding curvature sphere
$\Sigma$ is constant.  Note that $\Sigma$ is a hypersphere obtained by intersecting $S^{2l-1}$ with a hyperplane
$\pi$ that is tangent to $V_2 (C_{m-1})$ along $S$.  The image $T_{\alpha,\beta} (\pi)$ is a 
hyperplane that is tangent
to $V^{\alpha,\beta}_2$ along the $\mu$-dimensional sphere $T_{\alpha,\beta} (S)$.  Since the hypersphere
$T_{\alpha,\beta} (\pi) \cap S^{2l-1}$ is tangent to $V^{\alpha,\beta}_2$ along $T_{\alpha,\beta} (S)$, it is a curvature
sphere of $V^{\alpha,\beta}_2$ with multiplicity $\mu$, and $T_{\alpha,\beta} (S)$ is the corresponding curvature surface.

Thus, we have a bijective correspondence between the curvature surfaces of $V_2 (C_{m-1})$ and those of 
$V^{\alpha,\beta}_2$, and the Dupin condition
is clearly satisfied on $V^{\alpha,\beta}_2$.  Therefore,
$V^{\alpha,\beta}_2$ is a proper Dupin submanifold with four principal curvatures, including $\lambda_4 = \infty$.\\

\noindent
{\bf Computing the Lie curvature}\\

With the principal curvature functions defined as in equation (\ref{eq:Lie-curv-lambda}) and using the fact that 
$\lambda_4 = \infty$, there is a unique 
Lie curvature function $\Psi$ defined on $B(V^{\alpha,\beta}_2)$ by
\begin{equation}
\label{eq:3.7.5-a}
\Psi = \frac{\lambda_1 - \lambda_2}{\lambda_1 - \lambda_3}.
\end{equation}

We next show that the Legendre lift of $V^{\alpha,\beta}_2$ is not Lie equivalent to the Legendre lift of an isoparametric
hypersurface in $S^{2l-1}$ by showing that the Lie curvature 
$\Psi$ does not equal $1/2$ at some points of
the unit normal bundle $B(V^{\alpha,\beta}_2)$, as required for the Legendre lift of an isoparametric hypersurface.  
Moreover, we will show that the Lie curvature is not constant on $B(V^{\alpha,\beta}_2)$.

To compute the functions $\lambda_1 < \lambda_2 < \lambda_3$, we first note that
\begin{displaymath}
V^{\alpha,\beta}_2 \subset f^{-1}(0) \cap g^{-1}(0),
\end{displaymath}
where $f$ and $g$ are the real-valued functions defined on $S^{2l-1}$ by
\begin{equation}
\label{eq:def-f-g}
f(u,v) = \langle \frac{- \beta}{2 \alpha} u, u \rangle + \langle \frac{\alpha}{2 \beta} v, v \rangle, \quad g(u,v) = - \langle u, v \rangle.
\end{equation}
Thus, the gradients,
\begin{displaymath}
\xi = \left(\frac{- \beta}{\alpha} u, \frac{\alpha}{\beta} v \right), \quad \eta = (-v, -u), 
\end{displaymath}
of $f$ and $g$ are two unit normal vector fields on $V^{\alpha,\beta}_2$.  

By the construction of the OT-FKM-hypersurfaces, one can show that $l > m+1$ for the OT-FKM-hypersurfaces (see
Theorem 4.39 in the book \cite[pp. 105--106]{Cec1}), so we can choose $x,y \in {\bf R}^l$
such that
\begin{eqnarray}
\label{eq:3.7.6}
|x| & = & \alpha, \quad \langle x, u \rangle = 0, \quad \langle x, v \rangle = 0, \quad \langle x, E_i v \rangle = 0, \quad 1 \leq i \leq m-1, \nonumber \\
|y| & = & \beta, \quad \langle y, u \rangle = 0, \quad \langle y, v \rangle = 0, \quad \langle y, E_i u\rangle = 0, \quad 1 \leq i \leq m-1. \nonumber 
\end{eqnarray}
We define three curves,
\begin{eqnarray}
\label{eq:3.7.7}
\gamma (t) = (\cos t \ u + \sin t \ x, v), \quad \delta (t) = (u, \cos t \ v + \sin t \ y),\\
\varepsilon (t) = (\cos t \ u + \frac{\alpha}{\beta} \sin t \ v, - \frac{\beta}{\alpha} \sin t \ u + \cos t \ v ).\nonumber 
\end{eqnarray}
It is straightforward to check that each of these curves lies on $V^{\alpha,\beta}_2$ and goes through the point
$(u,v)$ when $t=0$.  Along $\gamma$, the normal vector $\xi$ is given by
\begin{displaymath}
\xi(t) = \left( - \frac{\beta}{\alpha} (\cos t \ u + \sin t \ x), \frac{\alpha}{\beta} v \right).
\end{displaymath}
Thus, the initial velocity vector to $\xi(t)$ satisfies
\begin{displaymath}
\overrightarrow{\xi} (0) = \left( - \frac{\beta}{\alpha} x, 0 \right) = - \frac{\beta}{\alpha} \overrightarrow{\gamma} (0).
\end{displaymath}
So $X = (x,0) = \overrightarrow{\gamma} (0)$ is a principal vector of $A_\xi$ at $(u,v)$ with corresponding principal curvature 
$\beta / \alpha$.

Similarly, $Y = (0,y) = \overrightarrow{\delta} (0)$ is a principal vector of $A_\xi$ at $(u,v)$ with corresponding principal curvature 
$- \alpha / \beta$.  Finally, along the curve $\varepsilon$, we have
\begin{displaymath}
\xi(t) = \left( - \frac{\beta}{\alpha} \left(\cos t \ u + \frac{\alpha}{\beta} \sin t \ v \right), 
\frac{\alpha}{\beta} \left(- \frac{\beta}{\alpha} \sin t \ u + \cos t \ v \right) \right).
\end{displaymath}
Then $\overrightarrow{\xi} (0) = (-v,-u) = \eta$, which is normal to $V^{\alpha,\beta}_2$ at $(u,v)$.  
Thus, we have $A_\xi Z = 0$,
for $Z = \overrightarrow{\varepsilon} (0)$, and $Z$ is a principal vector with corresponding principal curvature zero.  Therefore, at the point
$\xi (u,v)$ in $B(V^{\alpha,\beta}_2)$, there are four principal curvatures written in ascending order as in equation (\ref{eq:Lie-curv-lambda}) (recall that $\alpha$ and $\beta$ are positive),
\begin{equation}
\label{eq:pc-V-alpha-beta-1}
\lambda_1 = - \frac{\alpha}{\beta}, \quad \lambda_2 = 0, \quad \lambda_3 = \frac{\beta}{\alpha}, \quad \lambda_4 = \infty.
\end{equation}
At this point, the Lie curvature $\Psi$ is 
\begin{equation}
\label{eq:Lie-c-lambda-1}
\Psi =  \frac{\lambda_1 - \lambda_2}{\lambda_1 - \lambda_3} = \frac{- \alpha / \beta}{(- \alpha / \beta - \beta / \alpha )} = \alpha^2.
\end{equation}
Since $\alpha^2 \neq 1/2$, the Legendre lift of $V^{\alpha,\beta}_2$ is not 
Lie equivalent to an isoparametric
hypersurface.

Furthermore, since $A_{- \xi} = - A_{\xi}$, the principal curvatures at the point $- \xi (u,v)$ in $B(V^{\alpha,\beta}_2)$ are the negatives of those given in equation (\ref{eq:pc-V-alpha-beta-1}).
Thus, since the smooth principal curvature functions
$\lambda_1$, $\lambda_2$, $\lambda_3$ defined on $B(V^{\alpha,\beta}_2)$ by equation (\ref{eq:Lie-curv-lambda}) satisfy
$\lambda_1 < \lambda_2 < \lambda_3$, we have
\begin{equation}
\label{eq:pc-V-alpha-beta-2}
\lambda_1 = - \frac{\beta}{\alpha}, \quad \lambda_2 = 0, \quad \lambda_3 = \frac{\alpha}{\beta}, \quad \lambda_4 = \infty.
\end{equation}
at this point $- \xi (u,v)$ in $B(V^{\alpha,\beta}_2)$. Thus, at this point $- \xi (u,v)$, we have the Lie curvature,
\begin{equation}
\label{eq:Lie-c-lambda-2}
\Psi =  \frac{\lambda_1 - \lambda_2}{\lambda_1 - \lambda_3} = \frac{- \beta / \alpha}{(- \beta / \alpha - \alpha / \beta )} = \beta^2.
\end{equation}
Since $\beta^2 \neq \alpha^2$, the Lie curvature function $\Psi$ is not constant on $B(V^{\alpha,\beta}_2)$.

To obtain a compact proper Dupin hypersurface
in $S^{2l-1}$ with four principal curvatures that is
not Lie equivalent to an isoparametric hypersurface, one simply takes a tube $M$ over 
$V^{\alpha,\beta}_2$ in $S^{2l-1}$ of sufficiently small radius so that the tube is an embedded hypersurface.

In a related result,
Miyaoka \cite[Corollary 8.3, p. 252]{Mi3} proved that if
the Lie curvature $\Psi$ is constant on a compact, connected proper Dupin hypersurface with four principal curvatures,
then, in fact, $\Psi = 1/2$ on the hypersurface.\\

\noindent
{\bf Construction of the Miyaoka-Ozawa Examples}\\

Next we handle the construction of the examples due to Miyaoka and Ozawa
\cite{MO}.  The key ingredient here is the Hopf
fibration of $S^7$ over $S^4$.  Let ${\bf R}^8 = {\bf H} \times {\bf H}$, where ${\bf H}$ is the skew field of
quaternions.  
The Hopf fibering of the unit sphere $S^7$ in ${\bf R}^8$ over $S^4$ is given by
\begin{equation} 
\label{eq:3.7.9}
h(u,v) = (2 u \bar{v}, |u|^2 - |v|^2), \quad u,v \in {\bf H},
\end{equation}
where $\bar{v}$ is the conjugate of $v$ in ${\bf H}$.
One can easily compute that the image of $h$ lies in the unit sphere $S^4$ in the Euclidean space 
${\bf R}^5 = {\bf H} \times {\bf R}$.

Before beginning the construction of Miyaoka and Ozawa, we recall some facts about the Hopf fibration.  Suppose
$(w,t) \in S^4$, with $t \neq 1$, i.e., $(w,t)$ is not the point $(0,1)$.  We want to find the inverse image
of $(w,t)$ under $h$.  Suppose that
\begin{equation} 
\label{eq:3.7.10}
2 u \bar{v} = w, \quad |u|^2 - |v|^2 = t.
\end{equation}
Multiplying the first equation in \eqref{eq:3.7.10} by $v$ on the right, we obtain
\begin{equation} 
\label{eq:3.7.11}
2 u |v|^2 = wv, \quad 2u = \frac{w}{|v|} \frac{v}{|v|}.
\end{equation}
Since $|u|^2 + |v|^2 =1$, the second equation in \eqref{eq:3.7.10} yields
\begin{equation} 
\label{eq:3.7.12}
|v|^2 = (1-t)/2.
\end{equation}
If we write $z = v/|v|$, then $z \in S^3$, the unit sphere in ${\bf H} = {\bf R}^4$.  Then equations \eqref{eq:3.7.11} and
\eqref{eq:3.7.12} give
\begin{equation} 
\label{eq:3.7.13}
u = \frac{wz}{\sqrt{2(1-t)}}, \quad v = \sqrt{(1-t)/2}\  z, \quad z \in S^3.
\end{equation}
Thus, if $U$ is the open set $S^4 - \{(0,1)\}$, then $h^{-1}(U)$ is diffeomorphic to $U \times S^3$ by the formula
\eqref{eq:3.7.13}.  Of course, the second equation in \eqref{eq:3.7.10} shows that $h^{-1}(0,1)$ is just the 3-sphere
in $S^7$ determined by the equation $v=0$.  

We can find a similar local trivialization containing these points
with $v=0$ by beginning the process above with multiplication of equation \eqref{eq:3.7.10} by $\bar{u}$ on the left,
rather than by $v$ on the right.  As a consequence of this local triviality, if $M$ is an embedded submanifold
in $S^4$ which does not equal all of $S^4$, then $h^{-1}(M)$ is diffeomorphic to $M \times S^3$.  Finally, recall
that the Euclidean inner product on the space ${\bf R}^8 = {\bf H} \times {\bf H}$ is given by
\begin{equation} 
\label{eq:3.7.14}
(a,b)\cdot (u,v) = \Re (\bar{a} u + \bar{b} v),
\end{equation}
where $\Re w$ denotes the real part of the quaternion $w$.

The examples of Miyaoka and Ozawa all arise as inverse images under $h$ of 
proper Dupin hypersurfaces 
in $S^4$.  The
proof that these examples are proper Dupin is accomplished by first showing that they are 
taut, as defined in Section \ref{sec:1}. 
Thus, we begin with the following.

\begin{theorem}
\label{thm:3.7.1} 
Let $M$ be a compact, connected submanifold of $S^4$.  If $M$ is taut in $S^4$, then $h^{-1}(M)$ is taut in $S^7$.
\end{theorem}
\begin{proof}
Since both $M$ and $h^{-1}(M)$ lie in spheres, tautness of $h^{-1}(M)$ in $S^7$ is equivalent to tightness of
$h^{-1}(M)$ in ${\bf R}^8$, i.e., every nondegenerate linear 
height function in ${\bf R}^8$ has the minimum number
of critical points.  We write linear height functions in ${\bf R}^8$ in the form
\begin{equation} 
\label{eq:3.7.15}
f_{ab} (u,v) = \Re (au + bv) = (\bar{a}, \bar{b}) \cdot (u,v), \quad (a,b) \in S^7.
\end{equation}
This is the height function in the direction $(\bar{a}, \bar{b})$.  We want to determine when the point $(u,v)$
is a critical point of $f_{ab}$.  Without loss of generality, we may assume that $(u,v)$ lies in a local
trivialization of the form \eqref{eq:3.7.13} when making local calculations.  Let $x = (w,t)$ be a point of
$M \subset S^4$, and let 
\begin{displaymath}
(x,z) = (w,t,z)
\end{displaymath}
be a point in the fiber $h^{-1}(x)$.  The tangent space to $h^{-1}(M)$
at $(x,z)$ can be decomposed as $T_x M \times T_z S^3$.  We first locate the critical points of the restriction of 
$f_{ab}$ to the fiber through $(x,z)$.  By equations \eqref{eq:3.7.13} and \eqref{eq:3.7.15}, we have
\begin{eqnarray} 
\label{eq:3.7.16}
f_{ab} (w,t,z) & = & \Re \left(\frac{awz}{\sqrt{2(1-t)}} + bz \sqrt{(1-t)/2} \right) \\
& = & \Re (\alpha (w,t) z) = \alpha (w,t) \cdot \bar{z}, \nonumber
\end{eqnarray}
where
\begin{displaymath}
\alpha (w,t) = \frac{aw}{\sqrt{2(1-t)}} + b \sqrt{(1-t)/2}.
\end{displaymath}
This defines the map $\alpha$ from $S^4$ to ${\bf H}$.  If $Z$ is any tangent vector to $S^3$ at $z$, we write
$Zf_{ab}$ for the derivative of $f_{ab}$ in the direction $(0,Z)$.  Then
\begin{equation} 
\label{eq:3.7.17}
Zf_{ab} = \alpha (w,t) \cdot \bar{Z}
\end{equation}
at $(x,z)$.  Now there are two cases to consider.  First, if $\alpha (w,t) \neq 0$, then in order to have 
$Zf_{ab} = 0$ for all $Z \in T_z S^3$, we must have
\begin{equation} 
\label{eq:3.7.18}
\bar{z} = \pm \frac{\alpha (w,t)}{|\alpha (w,t)|}.
\end{equation}
So the restriction of $f_{ab}$ to the fiber has exactly two critical points with corresponding values
\begin{equation} 
\label{eq:3.7.19}
\pm |\alpha (w,t)|.
\end{equation}
The second case is when $\alpha (w,t) = 0$.  Then the restriction of $f_{ab}$ to the fiber is identically zero
by equation \eqref{eq:3.7.16}.  In both cases the function,
\begin{displaymath}
g_{ab} (w,t) = |\alpha (w,t)|^2,
\end{displaymath}
satisfies the equation
\begin{displaymath}
g_{ab} (w,t) = f_{ab}^2 (w,t,z),
\end{displaymath}
at the critical point.  The key in relating this fact to information about the submanifold $M$ is to note that
\begin{eqnarray} 
\label{eq:3.7.20}
g_{ab} (w,t) & = & |\alpha (w,t)|^2 = \frac{1}{2} \Re \{2 a \bar{b} w + (|a|^2 - |b|^2)t\} + \frac{1}{2}(|a|^2 + |b|^2)
\nonumber \\
 & = & \frac{1}{2} + \frac{1}{2} ((w,t) \cdot (2 \bar{a} b, |a|^2 - |b|^2)) \\
& = & \frac{1}{2} + \frac{1}{2} \ell_{ab} (w,t), \nonumber
\end{eqnarray}
where $\ell_{ab}$ is the linear height function on ${\bf R}^5$ in the direction
\begin{displaymath}
(2\bar{a} b, |a|^2 - |b|^2) = h (\bar{a}, \bar{b}).
\end{displaymath}
This shows that $g_{ab} (w,t) = 0$ if and only if $(w,t) = - h(\bar{a}, \bar{b})$.  Thus, if $- h(\bar{a}, \bar{b})$
is not in $M$, the restriction of $f_{ab}$ to each fiber has exactly two critical points of the form $(x,z)$, with $z$
as in equation \eqref{eq:3.7.18}.  For $X \in T_x M$, we write $Xf_{ab}$ for the derivative of $f_{ab}$ in the direction
$(X,0)$.  At the two critical points, we have
\begin{equation} 
\label{eq:3.7.21}
Xf_{ab} = d\alpha (X) \cdot \bar{z},
\end{equation}
\begin{equation} 
\label{eq:3.7.22}
Xg_{ab} = 2 d\alpha (X) \cdot \alpha (x) = \pm 2|\alpha (x)| (d\alpha (X)\cdot \bar{z}) = \pm 2|\alpha (X)| Xf_{ab}.
\end{equation}
Thus $(x,z)$ is a critical point of $f_{ab}$ if and only if $x$ is a critical point of $g_{ab}$.  By equation
\eqref{eq:3.7.20}, this happens precisely when $x$ is a critical point of $\ell_{ab}$.  We conclude that if 
$- h(\bar{a}, \bar{b})$ is not in $M$, then $f_{ab}$ has two critical points for every critical point of 
$\ell_{ab}$ on $M$.  The set of points $(a,b)$ in $S^7$ such that $- h(\bar{a}, \bar{b})$ belongs to $M$ has measure
zero.  If $(a,b)$ is not in this set, then $f_{ab}$ has twice as many critical points as the height function
$\ell_{ab}$ on $M$.  Since $M$ is taut, every nondegenerate 
height function $\ell_{ab}$ has $\beta (M;{\bf Z}_2)$ critical
points on $M$, where $\beta (M;{\bf Z}_2)$ is the sum of the 
${\bf Z}_2$-Betti numbers of $M$.  
Thus, except for a set of measure zero, every height function $f_{ab}$
has $2 \beta (M;{\bf Z}_2)$ critical points on $h^{-1}(M)$.  Since $h^{-1}(M)$ is diffeomorphic to $M \times S^3$,
we have 
\begin{displaymath}
\beta (h^{-1}(M);{\bf Z}_2) = \beta (M \times S^3;{\bf Z}_2) = 2 \beta (M;{\bf Z}_2).
\end{displaymath}
Thus, $h^{-1}(M)$ is taut in $S^7$. 
\end{proof}

We next use Theorem \ref{thm:3.7.1} to show that the inverse image under $h$ of a compact 
proper Dupin submanifold in
$S^4$ is proper Dupin.  Recall that a submanifold $M$ of codimension greater than one is proper Dupin if the Legendre
lift of $M$ is proper Dupin, as defined in Section \ref{sec:2}.

A taut submanifold is always Dupin,
but it may not be proper Dupin, as shown by Pinkall \cite{P5}, and by
Miyaoka \cite{Mi2}, independently,
for hypersurfaces,
i.e., the number of distinct principal curvatures may not be constant on the 
unit normal bundle $B(M)$.  

Ozawa \cite{Oz} proved that a taut 
submanifold
$M \subset S^n$  is proper Dupin if and only if every
connected component of a critical set of a linear height function on $M$ is a point or is homeomorphic to a sphere
of some dimension $k$.  (See also Hebda \cite{Heb3}.)  
This result is a key fact in the proof of the following theorem. 

\begin{theorem}
\label{thm:3.7.2} 
Let $M$ be a compact, connected proper Dupin
submanifold embedded in $S^4$.  Then $h^{-1}(M)$ is a proper Dupin
submanifold in $S^7$.
\end{theorem}
\begin{proof}
As we have noted before, Thorbergsson
\cite{Th1} proved that a compact proper Dupin hypersurface embedded in $S^n$
is taut, and Pinkall
\cite{P5} extended this result to the case where $M$ has codimension greater than one and the
number of distinct principal curvatures is constant on the unit normal 
bundle $B(M)$.  Thus, our $M$ is taut in
$S^4$, and therefore $h^{-1}(M)$ is taut in $S^7$ by Theorem \ref{thm:3.7.1}.  To complete the proof of the theorem,
we need to show that each connected component of a critical set of a height function $f_{ab}$ on $h^{-1}(M)$ is
a point or a sphere.

We use the same notation as in the proof of Theorem \ref{thm:3.7.1}.  Now suppose that $(x,z)$ is a critical point of
$f_{ab}$.  For $X \in T_xM$, we compute from equation \eqref{eq:3.7.16} that
\begin{equation} 
\label{eq:3.7.23}
Xf_{ab} = d\alpha (X) \cdot \bar{z}.
\end{equation}
From \eqref{eq:3.7.22}, we see that $Xg_{ab}$ also equals zero, and the argument again splits into two cases, depending
on whether or not $g_{ab}(x)$ is zero.  If $g_{ab}(x)$ is nonzero, then there are two critical points of 
$f_{ab}$ on the fiber $h^{-1}(x)$.  Thus a component in $h^{-1}(M)$ of the critical set of $f_{ab}$ through $(x,z)$
is homeomorphic to the corresponding component of the critical set containing $x$ of the linear function
$\ell_{ab}$ on $M$.  Since $M$ is proper Dupin, such a component is a point or a sphere.

The second case is when $g_{ab}(x) = f_{ab}^2 (x,z) = 0$.  As we have seen, this happens only if 
$x = - h(\bar{a}, \bar{b})$.  In that case, $x$ is an isolated absolute minimum of the height function $\ell_{ab}$.
Thus, the corresponding component of the critical set of $f_{ab}$ through $(x,z)$ lies in the fiber $h^{-1}(x)$, 
which is diffeomorphic to $S^3$.  From equation \eqref{eq:3.7.23}, we see that this component of the critical set
consists of those points $(x,y)$ in the fiber such that $\bar{y}$ is orthogonal to $d \alpha (X)$, for all
$X \in T_xM$.  We know that 
\begin{equation} 
\label{eq:3.7.24}
g_{ab}(x) = \frac{1}{2} + \frac{1}{2} \ell_{ab} (x),
\end{equation}
and $x$ is an isolated critical point of $\ell_{ab}$ on $M$.  The tautness of $M$ and the results of Ozawa \cite{Oz}
imply that $x$ is a nondegenerate critical point of $\ell_{ab}$, since the component of the critical set
of a height function
containing a degenerate critical point must be a sphere of dimension greater than zero.
By equation \eqref{eq:3.7.24}, $x$ is also a nondegenerate critical point of $g_{ab}$, and so the Hessian $H(X,Y)$
of $g_{ab}$ is nondegenerate at $x$.  Since $\alpha (x) = 0$, we compute that for $X$ and $Y$ in $T_xM$,
\begin{displaymath}
H(X,Y) = 2 d \alpha (X) \cdot d \alpha (Y).
\end{displaymath}
Hence, $d \alpha$ is nondegenerate at $x$, and the 
rank of $d \alpha$ is the dimension of $M$.  From this it follows
that the component of the critical set of $f_{ab}$ through $(x,z)$ is a sphere in $h^{-1}(x)$ of dimension
$(3 - \dim M)$.  Therefore, we have shown that every component of the critical set of a linear height
function $f_{ab}$ on $h^{-1}(M)$ is homeomorphic to a point or a sphere.  Thus, $h^{-1}(M)$ is proper Dupin.
\end{proof}

Next, we relate the principal curvatures of $h^{-1}(M)$ to those of $M$.

\begin{theorem}
\label{thm:3.7.3} 
Let $M$ be a compact, connected proper Dupin hypersurface embedded in $S^4$ with $g$ principal curvatures.  
Then the proper Dupin hypersurface $h^{-1}(M)$ in $S^7$ has $2g$ principal curvatures.  Each principal curvature,
\begin{displaymath}
\lambda = \cot \theta, \quad 0 < \theta < \pi,
\end{displaymath}
of $M$ at a point $x \in M$ yields two principal curvatures of 
$h^{-1}(M)$ at points in $h^{-1}(x)$ with values
\begin{displaymath}
\lambda^{+} = \cot (\theta / 2), \quad \lambda^{-} = \cot ((\theta + \pi) / 2).
\end{displaymath}
\end{theorem}
\begin{proof}
A principal curvature $\lambda = \cot \theta$ of a hypersurface $M$ at $x$ corresponds to a 
focal point at oriented distance
$\theta$ along the normal geodesic to $M$ at $x$. 
(See, for example, Cecil--Ryan \cite[p. 127]{CR7}.)  A point $(x,z)$
in $h^{-1}(M)$ is a critical point of $f_{ab}$ if and only if $(\bar{a}, \bar{b})$ lies along the normal geodesic to
$h^{-1}(M)$ at $(x,z)$.  The critical point is degenerate if and only if $(\bar{a}, \bar{b})$ is a focal point of
$h^{-1}(M)$ at $(x,z)$.  Note further that $(x,z)$ is a degenerate critical point of $f_{ab}$ if and only if $x$
is a degenerate critical point of $\ell_{ab}$.  This follows from the fact that both embeddings are taut, and the
dimensions of the components of the critical sets agree by Theorem \ref{thm:3.7.2}.  The latter claim holds even
when $x = - h(\bar{a}, \bar{b})$, since the fact that $M$ has dimension three implies that the critical point
$(x,z)$ of $f_{ab}$ is isolated.  Thus, $(\bar{a}, \bar{b})$ is a focal point of $h^{-1}(M)$ if and only if
$h(\bar{a}, \bar{b})$ is a focal point of $M$.  Suppose now that $(\bar{a}, \bar{b})$ lies along the normal geodesic to
$h^{-1}(M)$ at $(x,z)$ and that $f_{ab}(x,z) = \cos \phi$.  Then by equation \eqref{eq:3.7.20},
\begin{displaymath}
g_{ab}(x) = \frac{1}{2} + \frac{1}{2} \ell_{ab} (x) = \frac{1}{2} + \frac{1}{2} \cos \theta,
\end{displaymath}
where $\theta$ is the distance from $h(\bar{a}, \bar{b})$ to $x$.  Since $(x,z)$ is a critical point of $f_{ab}$,
we have $g_{ab}(x) = f_{ab}^2 (x,z)$. Thus,
\begin{displaymath}
\frac{1}{2} + \frac{1}{2} \cos \theta = \cos^2 \phi = \frac{1}{2} + \frac{1}{2} \cos 2 \phi,
\end{displaymath}
and so $\cos \theta = \cos 2 \phi$.  This means that under the map $h$, the normal geodesic to 
$h^{-1}(M)$ at $(x,z)$ double covers the normal geodesic to $M$ at $x$, since the points corresponding to the
values $\phi = \theta / 2$ and $\phi = (\theta + \pi) / 2$ are mapped to the same point by $h$.  In particular,
a focal point corresponding to a principal curvature $\lambda = \cot \theta$ on the normal geodesic to $M$ at $x$
gives rise to two focal points on the normal geodesic to $h^{-1}(M)$ at $(x,z)$ with corresponding principal curvatures
\begin{displaymath}
\lambda^{+} = \cot (\theta / 2), \quad \lambda^{-} = \cot ((\theta + \pi) / 2).
\end{displaymath}
This completes the proof of the theorem.
\end{proof}

We now construct the examples of Miyaoka and Ozawa.  Recall that a compact proper 
Dupin hypersurface $M$ in $S^4$ with two principal curvatures must be a cyclide of Dupin, 
that is, the image under
a M\"{o}bius transformation of $S^4$ of a standard product of spheres,
\begin{displaymath}
S^1 (r) \times S^2 (s) \subset S^4 (1) \subset {\bf R}^5, \quad r^2 + s^2 = 1.
\end{displaymath}
A conformal, non-isometric image of an isoparametric cyclide does not have constant principal curvatures.  Similarly,
a compact proper Dupin
hypersurface in $S^4$ with three principal curvatures must be 
Lie equivalent to an isoparametric
hypersurface in $S^4$ with three principal curvatures, but it need not have constant principal curvatures itself.

\begin{corollary}
\label{cor:3.7.4} 
Let $M$ be a non-isoparametric compact, connected proper Dupin hypersurface embedded in $S^4$ with $g$ principal 
curvatures, where $g= 2$ or 3.  Then $h^{-1}(M)$ is a compact, connected proper Dupin hypersurface in $S^7$
with $2g$ principal curvatures that does not have constant Lie curvatures.  Therefore, it is not Lie equivalent to an isoparametric hypersurface in $S^7$.
\end{corollary}
\begin{proof}
Suppose that $\lambda = \cot \theta$ and $\mu = \cot \alpha$ are two distinct nonconstant principal curvature
functions on $M$.  Let
\begin{eqnarray}
\lambda^{+} & = & \cot (\theta / 2), \quad \lambda^{-} = \cot ((\theta + \pi) / 2), \nonumber \\
\mu^{+} & = & \cot (\alpha / 2), \quad \mu^{-} = \cot ((\alpha + \pi) / 2), \nonumber
\end{eqnarray}
be the four distinct principal curvature functions on $h^{-1}(M)$ induced from $\lambda$ and $\mu$.  Then the Lie curvature
\begin{displaymath}
\Psi = \frac{(\lambda^{+} - \lambda^{-})(\mu^{+} - \mu^{-})}{(\lambda^{+} - \mu^{-})(\mu^{+} - \lambda^{-})}
= \frac{2}{1 + \cos (\theta - \alpha)}, 
\end{displaymath}
is not constant on $h^{-1}(M)$, and therefore $h^{-1}(M)$ is not Lie equivalent to an isoparametric hypersurface
in $S^7$. 
\end{proof}

Certain parts of the construction of Miyaoka and Ozawa
are also valid if ${\bf H}$ is replaced by the Cayley
numbers
or a more general Clifford algebra.  
See the paper of Miyaoka and Ozawa  \cite{MO} for a discussion of this point.\\

Miyaoka \cite{Mi3} proved that the assumption that the Lie curvature $\Psi$ is constant on a 
compact, connected proper Dupin hypersurface $M$ in $S^n$ with four principal curvatures, 
together with an additional assumption regarding the intersections
of leaves of the various principal foliations, implies that $M$ is Lie equivalent to an isoparametric
hypersurface.  (Miyaoka \cite{Mi4} also proved a similar result for compact proper Dupin hypersurfaces with six principal 
curvatures.) 
Thorbergsson\cite{Th1} (see also Stolz \cite{Stolz}
and Grove--Halperin \cite{GH}) proved
that for a compact proper Dupin hypersurface 
in $S^n$ with four principal curvatures, 
the multiplicities
of the principal curvatures must satisfy $m_1 = m_2, m_3 = m_4$, when the principal
curvatures are appropriately ordered.

Then Cecil, Chi and Jensen \cite{CCJ2} used a different approach than Miyaoka
to prove that if $M$ is a
compact, connected proper Dupin hypersurface in $S^n$ with four principal curvatures whose multiplicities
satisfy $m_1 = m_2 \geq 1$, $m_3 = m_4 =1$ and constant Lie curvature $\Psi$, then $M$ is Lie equivalent to an 
isoparametric hypersurface.  Thus, Miyaoka's additional assumption regarding the intersections
of leaves of the various principal foliations is not needed in that case.  

It remains an open question whether
Miyaoka's additional assumption can be removed in the case where $m_3 = m_4$ is also allowed to be greater than one,
although this has been conjectured to be true by 
Cecil and Jensen \cite[pp. 3--4]{CJ3}.

\noindent Thomas E. Cecil

\noindent Department of Mathematics and Computer Science

\noindent College of the Holy Cross

\noindent Worcester, MA 01610

\noindent email: tcecil@holycross.edu\\

\end{document}